\documentclass[conference]{IEEEtran}
\IEEEoverridecommandlockouts

\usepackage{mathrsfs}
\usepackage{graphics} 

\usepackage{amsmath} 
\usepackage{amssymb}  
\usepackage{amsthm}
\usepackage{cite}
\usepackage{bm}
\usepackage{acronym}
\usepackage{paralist}
\usepackage{float}
\usepackage{color}
\usepackage{epstopdf}
\usepackage{multicol}
\usepackage{tikz}
\usepackage{hyperref}
\usepackage{graphicx}
\usepackage{soul}
\usepackage{mathtools}

\makeatletter
\hypersetup{colorlinks=true}
\AtBeginDocument{\@ifpackageloaded{hyperref}
  {\def\@linkcolor{blue}
  \def\@anchorcolor{red}
  \def\@citecolor{red}
  \def\@filecolor{red}
  \def\@urlcolor{red}
  \def\@menucolor{red}
  \def\@pagecolor{red}
\begingroup
  \@makeother\`%
  \@makeother\=%
  \edef\x{%
    \edef\noexpand\x{%
      \endgroup
      \noexpand\toks@{%
        \catcode 96=\noexpand\the\catcode`\noexpand\`\relax
        \catcode 61=\noexpand\the\catcode`\noexpand\=\relax
      }%
    }%
    \noexpand\x
  }%
\x
\@makeother\`
\@makeother\=
}{}}
\makeatother

\newtheorem{Theorem}{Theorem}

\newtheorem{Lemma}{Lemma}

\newtheorem{Remark}{Remark}

\newtheorem{Corollary}{Corollary}

\newtheorem{Definition}{Definition}

\newcommand{\bequ}{\begin{eqnarray}}
\newcommand{\eequ}{\end{eqnarray}}

\def\IR{{\mathbb R}}

\IEEEoverridecommandlockouts

\def\BibTeX{{\rm B\kern-.05em{\sc i\kern-.025em b}\kern-.08em
    T\kern-.1667em\lower.7ex\hbox{E}\kern-.125emX}}

\begin{document}

\title{Equivalence of Finite- and Fixed-time Stability to Asymptotic Stability}

\author{Kunal Garg, \IEEEmembership{Member, IEEE}
\thanks{The author is with the School for Engineering of Matter, Transport, and Energy at Arizona State University, Tempe, AZ, 85281. Email: \texttt{kgarg24@asu.edu}.}
}
\maketitle

\begin{abstract}
In this paper, we present new results on finite- and fixed-time convergence for dynamical systems using LaSalle-like invariance principles. In particular, we provide first and second-order non-smooth Lyapunov-like results for finite- and fixed-time convergence, thereby relaxing the requirement of existence a differentiable, positive definite Lyapunov function. Based on these findings, we show that a dynamical system whose equilibrium point is globally asymptotically stable can be modified through scaling so that the resulting dynamical system has a fixed-time stable equilibrium point. The results in this paper expand our understanding of various convergence rates and strengthen the hypothesis that all the convergence rates are interconnected through a suitable transformation. 
\end{abstract}

\section{Introduction}
The notions of asymptotic, exponential, finite- and fixed-time stability have been studied extensively in the past couple of decades, with focus on theoretical advancements and new applications of faster convergence notions \cite{bhat2005geometric,polyakov2012nonlinear,garg2020fixed,garg2022fixed,liu2023multiobjective,poveda2022fixed,zhao2023prescribed}.
It has been shown that faster convergence rates result in improved robustness against both vanishing and non-vanishing additive disturbances \cite{bhat2000finite,polyakov2012nonlinear}, so these faster notions of convergence are of particular importance for their applications, say, in control and optimization. As the focus on machine learning and methods for optimizing neural network training has grown, these notions of faster convergence have attracted significant attention in recent years \cite{budhraja2022breaking,ju2023neurodynamic,yang2025twonovel,ozbek2025survey}. 

However, one of the theoretical questions researchers have been trying to answer lately is: What is the relationship between various convergence notions? 
Such a question was answered in the context of control systems in \cite{bhat2005geometric}, where the authors proved that it is possible to design a continuous finite-time stabilizing controller for a controllable linear time-invariant system. In the context of optimization, similar equivalence results have been studied ever since the notion of finite and fixed-time stability became popular. Based on the results in \cite{cortes2006finite}, the authors in \cite{cortes2008discontinuous} showed that normalization of gradient flows, whose equilibrium points are asymptotically stable (AS), leads to finite-time stability (FTS). These results were further extended in \cite{garg2020fixed}, where novel modifications for gradient flows in the context of convex optimization were introduced that convert a system whose equilibrium point is exponentially stable (ES) to a system for which the same equilibrium is fixed-time stable (FxTS). This equivalence was further formalized in \cite{garg2022fixed} and generalized for non-smooth optimization problems, where it was shown that the modification is a natural result of temporal transformation from infinity to a fixed number. In the context of optimization, convexity or equivalent notions naturally help guarantee asymptotic convergence, and the focus remains on the modification of the dynamics to accelerate this convergence. As a result, these results are so far limited to specific classes of dynamical systems, such as gradient flows. More recent work focuses on \textit{smooth} Lyapunov analysis, with the aim of deriving sufficient conditions assuming the availability of a differentiable Lyapunov function or its construction. In \cite{ozaslan2024exponential}, the authors prove the equivalence of ES and FTS and FxTS, i.e., if an equilibrium point for a dynamical system of the form \eqref{eq: sys} is ES, then the equilibrium point for a scaled dynamical system, scaled properly with the norm of the dynamics function $f$, has the same equilibrium point FTS or even FxTS. 

In the context of non-smooth analysis, the authors in \cite{cortes2006finite} showed that normalized gradient flows, under the assumption of invariance of a compact set containing the equilibrium point, have the equilibrium point FTS. The results in \cite{cortes2006finite} utilized non-smooth analysis along with a variation of the invariance principle to provide conditions for FTS without resorting to a differentiable Lyapunov function. We generalize this idea further and, instead of restricting our discussion to only gradient flows, illustrate how all the notions, namely, AS, ES, FTS, and FxTS, are essentially equivalent under a proper transformation. 

The main contributions of this work are twofold. One, we present extensions of prior results on non-smooth Lyapunov-like conditions for FTS and FxTS, utilizing invariance principles based on global AS (GAS). In particular, we provide novel first and second-order non-smooth results for FTS and FxTS under the assumption of GAS for a general class of nonlinear dynamical systems. Then, we show that a general dynamical system, whose equilibrium is GAS, can be modified, through normalization or scaling, so that the same equilibrium point becomes FTS or FxTS for the modified system, thereby proving the equivalence of AS, FTS, and FxTS under a suitable geometric transformation of the system dynamics.

The rest of the paper is organized as follows. Section \ref{sec: math prelim} provides the relevant background and preliminary results. Section \ref{sec: new results} provides an extension of the existing results to provide a base for the main results on the equivalence of AS, ES, FTS, and FxTS, which are presented in Section \ref{sec: AS FTS equivalence}. Lastly, Section \ref{sec: conclusions} provides the conclusions and directions for future work. 

\section{Mathematical Preliminaries}\label{sec: math prelim}
\textit{Notations}: In the rest of the paper, $\IR$ denotes the set of real numbers and $\mathbb R_+$ denotes the set of non-negative real numbers. We use $\|\cdot\|_p$ to denote the $p-$norm, and $\|\cdot\|$ is used to denote the Euclidean norm. In addition, we write $\partial S$ for the boundary of the closed set $S$, $\textrm{int}(S)$ for its interior, and $|x|_S = \inf_{y\in S}\|x-y\|$ for the distance of the point $x\notin S$ from the set $S$. The Lie derivative of a continuously differentiable function $V:\mathbb R^n\rightarrow \mathbb R$ along a continuous vector field $f:\mathbb R^n\rightarrow\mathbb R^n$ at a point $x\in \mathbb R^n$ is denoted as $L_fV(x) \triangleq \frac{\partial V}{\partial x} f(x)$, while that of a locally Lipschitz function $V$ is defined as a set-valued map $\tilde L_fV(x) = \{a\in \mathbb R\; |\; \eta ^Tf(x) = a, \eta\in \partial V(x)\}$, where $\partial V(x) = \textnormal{co}\{\lim_{i\to +\infty}dV(x_i)\; |\; x_i\to x, x_i\notin Z\cup \Omega_V\}$ where $\Omega_V$ is the set of points where $V$ fails to be differentiable and $Z\subset\mathbb R^n$ is any set of measure zero. Note that when $\tilde L_fV$ is single-valued, it holds that $\tilde L_fV = \{L_fV\}$.

Consider the nonlinear system
\begin{align}\label{eq: sys}
\dot x = f(x), \quad x(0) = x_0,
\end{align}
where $x\in \mathbb R^n$ and $f: \mathbb R^n \rightarrow \mathbb R^n$ is continuous with $f(0)=0$, without loss of generality. We assume that there exists a unique solution for \eqref{eq: sys} for each $x(0)\in \mathbb R^n$ for all $t\geq 0$. The following results, from \cite{cortes2006finite}, presented here for the reader's benefit, form the base for the main results in this paper. In plain words, the following results establish the relationship between the time derivatives of a possibly non-differentiable function along the solutions of \eqref{eq: sys}, which serves as the basis for nonsmooth analysis.  

\begin{Lemma}\label{lemma: 1}\hspace{-0.1pt}\cite[Theorem 1]{cortes2006finite}
    Let $x:\mathbb R_+\to \mathbb R^n$ be a Fillipov solution of \eqref{eq: sys} and let $V:\mathbb R^n\to \mathbb R$ be a locally Lipschitz and regular function. Then, $V\circ x:\mathbb R_+\to \mathbb R$ is absolutely continuous and $\frac{d}{dt}V(x(t))\in \tilde L_fV(\phi(t))$ almost everywhere for $t\in \mathbb R_+$. Consequently, $\frac{d}{dt}V(x(t))\leq \max  \tilde L_fV(x(t))$ almost everywhere for $t\in \mathbb R_+$.
\end{Lemma}

\begin{Lemma}\label{lemma: 2}\hspace{-0.1pt}\cite[Proposition 2]{cortes2006finite}
    Let $x:\mathbb R_+\to \mathbb R^n$ be a Fillipov solution of \eqref{eq: sys} and let $V:\mathbb R^n\to \mathbb R$ be a locally Lipschitz and regular function. Assume that $\tilde L_fV$ is single-valued, Lipschitz, and regular. 
    Then, $\frac{d^2}{dt^2}V(x(t))$ exists almost everywhere and $\frac{d^2}{dt^2}V(x(t))\in \tilde L_f(\tilde L_fV(x(t)))$ almost everywhere for $t\in \mathbb R_+$. Consequently, $\frac{d^2}{dt^2}V(x(t))\geq \min \tilde L_f(\tilde L_fV(x(t)))$ almost everywhere for $t\in \mathbb R_+$. 
\end{Lemma}

Since the focus of this work is on convergence rates, we formally define the stability and convergence notions below. These definitions are borrowed from \cite{khalil2002nonlinear,bhat2000finite,polyakov2012nonlinear}.\footnote{Here, with slight abuse of notation and for brevity, we use $x(\cdot)$ to denote the solution of \eqref{eq: sys}.} 
\begin{Definition}
    The equilibrium point $x = 0$ for \eqref{eq: sys} is termed as
    \begin{enumerate}
        \item Lyapunov stable: if for $\epsilon>0$, there exists $\delta = \delta(\epsilon)>0$ such that $\|x(0)\|<\delta\implies \|x(t)\|<\epsilon$ for all $t\geq 0$. 
        \item Asymptotically stable: if it is Lyapunov stable and there exists $c>0$ such that $\|x(t)\|\to 0$ as $t\to \infty$ for all $\|x(0)\|<c$.
        \item Exponentially stable: if it is AS and there exists $k, \lambda>0$ such that $\|x(t)\|\leq \lambda \|x(0)\|e^{-kt}$ for all $t\geq 0$. 
        \item Finite-time stable: if it is Lyapunov stable and there exists a settling-time function $T:\mathbb R^n\to \mathbb R_+$ such that $T(x(0))<\infty$ for all $x(0)\in \mathbb R^n$ and $x(t)\to 0$ as $t\to T(x(0))$. 
        \item Fixed-time stable: if it is FTS and the settling-time function is bounded for all $x(0)\in \mathbb R^n$, i.e., $\lim_{\|x(0)\|\to \infty}T(x(0))<\infty$.
    \end{enumerate}
\end{Definition}

Lyapunov conditions for AS and ES can be found in \cite[Chapter 4]{khalil2002nonlinear}. Lyapunov conditions for FTS of the origin for system \eqref{eq: sys} were first presented in \cite{bhat2000finite}, which we review below. 

\begin{Lemma}\hspace{-0.1pt}\cite[Theorem 4.2]{bhat2000finite}\label{FTS Bhat}
Suppose there exists a positive definite function $V:\mathcal D\to\mathbb R$, where $\mathcal D\subset\mathbb R^n$ is a neighborhood of the origin, constants $c>0$ and $\alpha \in (0, 1)$, and an open neighborhood $\mathcal{V}\subseteq \mathcal{D}$ of the origin such that 
\begin{equation} \label{FTS Lyap}
    (D^+V)(x) + cV(x)^\alpha \leq 0, \; \forall x\in \mathcal{V}\setminus\{0\},
\end{equation}
where $D^+V$ denotes the upper-right Dini derivative of $V$. Then, the origin is an FTS equilibrium of \eqref{eq: sys}. Moreover, the settling time function $T$ satisfies $T(x(0)) \leq \frac{V(x(0))^{1-\alpha}}{c(1-\alpha)}$.
\end{Lemma}

The authors in \cite{polyakov2012nonlinear} presented the following result for fixed-time stability, where the time of convergence does not depend upon the initial condition. 

\begin{Theorem}\hspace{-0.1pt}\cite[Lemma 1]{polyakov2012nonlinear}\label{FxTS TH}
Suppose there exists a positive definite radially unbounded function $V$ for system \eqref{eq: sys} such that 
\begin{align}\label{eq: dot V FxTS old }
    (D^+V)(x) \leq -aV(x)^p-bV(x)^q, \quad \forall x\neq 0,
\end{align}
with $a,b>0$, $0<p<1$ and $q>1$. Then, the origin of \eqref{eq: sys} is FxTS with the settling time function satisfying
\begin{align}\label{eq: T bound old}
    T \leq \frac{1}{a(1-p)} + \frac{1}{b(q-1)}. 
\end{align}
\end{Theorem}

Note that while the above results allow the Lyapunov function to be non-differentiable, they still require the existence of a positive definite function that serves as a Lyapunov function. The purpose of this paper is to further relax this condition by utilizing invariance-like principles. 

\section{New results on FTS and FxTS}\label{sec: new results}
First, we review a variant of LaSalle's invariance principle from \cite{cortes2006finite} for the case of general nonlinear systems of the form \eqref{eq: sys} using a potentially non-smooth Lyapunov-like function. 

\begin{Lemma}\cite[Theorem 3]{cortes2006finite}\label{lemma: LaSalle nonsmooth}
Let $V:\mathbb R^n\to \mathbb R$ be a locally Lipschitz continuous and regular function and $S\subset\mathbb R^n$ be a compact and strongly invariant for \eqref{eq: sys} with $x(0)\in S$. Assume that $\max\tilde L_fV(x)\leq 0$ for all $x\in S$ and let $M = \{x\; |\; 0\in \tilde L_fV(x)\}$. Then, any solution of \eqref{eq: sys} starting from $x(0)$ converges to the largest invariant set $E$ in $M\cap S$. 
\end{Lemma}
A set $S\subset\mathbb R^n$ is termed as strongly invariant for \eqref{eq: sys} if for each initial condition $x(0)\in S$, every resulting maximal solution $x(\cdot)$ satisfies $x(t)\in S$ for all $t\geq 0$.  

\begin{Remark}
In Lemma \ref{lemma: LaSalle nonsmooth}, the function $V$ is not necessarily a Lyapunov function since it is not assumed to be positive definite. Furthermore, the result guarantees convergence to the largest invariant set where the derivative of the function $V$ is zero, which may not coincide with the points where the function $V$ itself is zero. This is the crucial difference between invariance-like principles and Lyapunov function-based results. 
\end{Remark}

The requirement of the existence of an invariant set for \eqref{eq: sys} can be fulfilled by assuming the function $V$ is lower-bounded, as $\dot V(x)\leq \max L_fV(x)\leq 0$, wherever $\dot V$ exists, implies that the level-sets of the function $V$ are invariant. The following corollary can be immediately stated based on these assumptions. 

\begin{Corollary}\label{cor: immediate extension}
    Let $V:\mathbb R^n\to \mathbb R$ be a locally Lipschitz continuous and regular function with compact sublevel sets, i.e., $\{x\; |\; V(x)\leq c\}$ is compact for $c\geq 0$. Assume that $\max\tilde L_fV(x)\leq 0$ for all $x$ and let $M = \{x\in \mathbb R^n \; |\; 0\in \tilde L_fV(x)\}$. Then, for each $x(0)$, there exists $S\subset\mathbb R^n$ that is invariant for \eqref{eq: sys} and any solution of \eqref{eq: sys} starting from $x(0)$ converges to the largest invariant set $E$ in $M\cap S$. 
\end{Corollary}
Another immediate extension is possible when the domain of attraction (see, e.g., \cite[pp. 122]{khalil2002nonlinear}) of an equilibrium point is known, as illustrated in the following result. 

\begin{Corollary}\label{cor: local AS with DoA}
    Assume that $x = 0$ is AS for \eqref{eq: sys} with $D\subset\mathbb R^n$ its domain of attraction. Let $V:\mathbb R^n\to \mathbb R$ be a locally Lipschitz continuous and regular function. Assume that $\max\tilde L_fV(x)\leq 0$ for all $x\in D$ and let $M = \{x\in D\; |\; 0\in \tilde L_fV(x)\}$. Then, for each $x(0)\in D$, there exists $S\subset D$ that is invariant for \eqref{eq: sys} and any solution of \eqref{eq: sys} starting from $x(0)$ converges to the largest invariant set $E$ in $M\cap S$. 
\end{Corollary}

\begin{Remark}
    Corollary \ref{cor: immediate extension} and \ref{cor: local AS with DoA} provide a method of constructing the invariant set $S$ in Lemma \ref{lemma: LaSalle nonsmooth} under the assumptions of compact sub-level sets of the function $V$ and a known domain of attraction. We use these particular cases in the main results of the paper presented next. 
\end{Remark}

For FTS, the authors in \cite{cortes2006finite} assume that the continuous Lyapunov-like function $V$ is bounded from below over the compact set $S$ and satisfies $\max \tilde L_fV(x)<-\epsilon$ for some $\epsilon>0$ for all $x\in S\setminus M$, guaranteeing existence of a finite time $T$ for any initial condition $x(0)\in S$ such that $\lim_{t \uparrow T}x(t)\in M$. 
We propose a new condition that provides a natural way of obtaining conditions for FxTS as well. 

\begin{Theorem}\label{thm: FTS LaSalle}
Assume 
that the solutions of \eqref{eq: sys} are uniquely determined for all initial conditions. Let $V:\mathbb R^n\to \mathbb R$ be a Lipschitz continuous and regular function, whose sublevel sets are compact,   
and let $M = \{x \in \mathbb R^n \; |\; 0\in \tilde L_fV(x)\}$. 
Consider a monotonically increasing continuous scalar function $\phi:\mathbb R\to\mathbb R$ satisfying $r\phi(r)> 0$ for all $r\neq 0$, $\phi(0) = 0$, and
\begin{align*}
    \int_{0}^{V_0}\frac{dV}{\phi(V)}<\infty \quad \forall V_0\in \mathbb R.
\end{align*}
If $\max \tilde L_fV(x)\leq -\phi(V(x))$ for all $x\notin M $, then 
for each $x(0)$ such that $V(x(0))>0$, there exists $S\subset\mathbb R^n$ that is invariant for \eqref{eq: sys} and the solutions of \eqref{eq: sys} starting from $x(0)$ converge to the largest invariance set $E\subset M\cap S$ within a finite time upper-bounded by $T(x(0))\leq \int_{0}^{V(x(0))}\frac{dV}{\phi(V)}$. 
Furthermore, if $V$ is radially unbounded w.r.t. $M$ and $\sup_{V_0\in \mathbb R}\int_{0}^{V_0}\frac{dV}{\phi(V)}=T_M<\infty$ for some $T_M>0$, then 
the settling-time function is upper-bounded by $T_M$ for all $x(0)$. 
\end{Theorem}
\begin{proof}
    Using monotonicity of the function $\phi$, we obtain that $\max \tilde L_fV(x)\leq 0$. Hence, from Corollary \ref{cor: immediate extension}, it follows that the trajectories of the system \eqref{eq: sys} converge to the set $E$.

    Next, we prove that this convergence happens within a finite time. Given the assumptions on the function $V$, per Lemma \ref{lemma: 1}, we have that $V\circ x$ is absolutely continuous, and hence, it holds that
    \begin{align*}
        \frac{d}{dt}V(x(t))\leq \max \tilde L_fV(x(t)) \leq -\phi(V(x(t))\leq 0,
    \end{align*}
    almost everywhere for $t$ such that $V(x(t))>0$. Hence, the sublevel set $S = \{x\; |\; V(x)\leq V(x(0))\}$ is invariant. 
    Using the comparison lemma (see, e.g., \cite[Lemma 3.4]{khalil2002nonlinear}), we have that
    \begin{align*}
        V(x(t))\leq v(t), \quad t\geq 0,
    \end{align*}
    where $v:\mathbb R\to \mathbb R$ is the unique solution of:
    \begin{align*}
        \dot v = -\phi(v), \quad v(0) = V(x(0)).
    \end{align*}
    Let $V_0 = V(x(0))$. Using the monotonicity of the function $\phi$, it holds that $\phi(v(t))\geq \phi(V(x(t)))>0$ for all $t\geq 0$ such that $V(x(t))>0$. Hence, along the solution $v(\cdot)$, we have:
    \begin{align*}
        & \int_{v(0)}^{v(t)}\frac{dv}{\phi(v)} = -t 
    \end{align*}
    The integral on the LHS above computes the time it takes for the solution $v(\cdot)$ to reach a point $v(t)$ starting from $v(0)$.  
    Consequently, the settling-time function, i.e., the time required for a solution to reach the origin for a given initial condition $v(0)$, is given as
    \begin{align*}
        \mathcal T(v(0)) & \coloneqq  \int_{0}^{v(0)}\frac{dv}{\phi(v)}=\int_{0}^{V_0}\frac{dV}{\phi(V)}<\infty,
    \end{align*}
    where we used the fact that the function $\mathcal T$ is bounded for $V_0 = V(x(0))\in \mathbb R$ per the assumption of the theorem. Hence, the solutions $v(\cdot)$ reach the value 0 within a finite time for any $v(0)\in \mathbb R$. Since $V(x(t))\leq v(t)$ for all $t\geq 0$, it holds that $\lim_{t \uparrow \mathcal T(V(x(0)))}V(x(t))= 0$. Since $\phi(r)\neq 0$ for $r\neq 0$ and $\max \tilde L_fV(x)\leq -\phi(V(x))$, it holds that the trajectories reach the largest invariant set $E\subset M\cap S$ within a finite time.  
    Fixed-time convergence follows when $\int_{0}^{V(x(0))}\frac{dv}{\phi(V)}$ is bounded by $T_M$ for all $V(x(0))\in \mathbb R$. 
\end{proof}

One of the benefits of the conditions of Theorem \ref{thm: FTS LaSalle} is that they allow higher flexibility in its application for proving finite-time (or fixed-time) convergence. The condition $\max \tilde{L}_fV(x)\leq -\epsilon$ is generally satisfied for \textit{normalized} dynamics where the norm of the derivative of the system state, $\|\dot x\|$, is constant. Such dynamics find applications in the context of convex optimization through normalized gradient flow (see \cite[Section 3]{cortes2006finite}) or in the context of control design through sliding-mode methods (see, e.g., \cite{song2016finite}). However, such normalized methods lead to \textit{chattering} behavior due to the discontinuity of the dynamics at the equilibrium point. The conditions in Theorem \ref{thm: FTS LaSalle} allow for continuous dynamics, enabling its application to a broader range of problems without leading to undesirable chattering behavior. 

\begin{Remark}
Note that the derivative conditions for FTS in \cite{bhat2000finite} become special case of Theorem \ref{thm: FTS LaSalle} with $\phi(s) = as^p$ with $a>$ and $0<p<1$, and those for FxTS in \cite{polyakov2012nonlinear} become special case with $\phi(s) = as^p+bs^q$ where $a, b>0$, $0<p<1$ and $q>1$. Remark 4 in \cite{garg2020fixed} explains how these two terms result in accelerated convergence and ultimately, FxTS. The main insight required to construct such a function is as follows. The integral $\int_0^{s_0}\frac{ds}{\phi(s)}$ is generally decomposed into two components: $\int_0^{1}\frac{ds}{\phi(s)} + \int_1^{s_0}\frac{ds}{\phi(s)}$. The first component is bounded when $\frac{\phi(s)}{s}\to \infty$ as $s\to 0$ and the second component is bounded when $\frac{\phi(s)}{s}\to \infty$ as $s\to \infty$. These properties can help choose the appropriate function $\phi$. 
\end{Remark}

Next, we present the extension of \cite[Theorem 5]{cortes2006finite}, using the derivative information of the system dynamics $f$. 
\begin{Theorem}\label{thm: FxTS LaSalle second order}
    Assume that $f$ in \eqref{eq: sys} is Lipschitz continuous over $\mathbb R^n\setminus \{0\}$. Let $V:\mathbb R^n\to \mathbb R$ be a continuously differentiable function with compact sublevel sets that is bounded from below with a Lipschitz gradient such that its Lie derivative along the system trajectories of \eqref{eq: sys} satisfies $L_fV(x)\leq 0$ for all $x$ and let $M = \{x\; |\; L_fV(x) = 0\}$. Then, for each $x(0)\in \mathbb R^n$, there exists $S\subset\mathbb R^n$ that is invariant for \eqref{eq: sys}. Further assume that $\min \tilde L_f(L_fV)(x)\geq \phi(-L_fV(x))$ for all $x\in S\setminus M$, where $\phi:\mathbb R_+\to\mathbb R_+$ is the same as in Theorem \ref{thm: FTS LaSalle}. Then, the solutions of \eqref{eq: sys} starting from any $x(0)\in \mathbb R^n$ converge to the largest invariant set $E\subset M\cap S$ within a finite time upper-bounded by $T\leq \int_0^{- L_fV(x(0))} \frac{dv}{\phi(v)}$. Furthermore, if $\sup_{r_0\in \mathbb R_+}\int_0^{r_0} \frac{dr}{\phi(r)}\leq T_M<\infty$, then the trajectories of \eqref{eq: sys} reach the set $E$ within a fixed time upper-bounded by $T_M$ for all $x(0)\in \mathbb R^n$. 
\end{Theorem}
\begin{proof}
    Since $V$ is continuously differentiable with Lipschitz gradient and the function $f$ is Lipschitz, it holds that a unique solution of \eqref{eq: sys} exists and the function $\tilde L_fV$ is Lipschitz, singleton, and regular and is given as $\tilde L_fV = L_fV$. Since $L_fV(x)\leq 0$, it holds that $S = \{x\; |\; V(x)\leq V(x(0))\}$ is invariant for \eqref{eq: sys}. 
    Hence, per Lemma \ref{lemma: 2}, it holds $\frac{d^2}{dt^2}V(x(t))\in \tilde L_f(\dot V(x(t)))$. Using the arguments from the proof of \cite[Theorem 5]{cortes2006finite}, defining $g(t) = -L_fV(x(t))$, we obtain that
    \begin{align*}
        \frac{d}{dt}g(t) \leq -\phi(-L_fV(x(t))) = -\phi(g(t)). 
    \end{align*}
    where $g$ is lower bounded by $0$ for all $t$ such that $x(t)\notin M$ and is upper bounded for all $t\geq 0$ since $L_fV(x)$ is continuous and hence, bounded in any compact domain $S\subset \mathbb R^n$. Using the same steps as in the proof of Theorem \ref{thm: FTS LaSalle}, we obtain that $g(t)\to 0$ as $t\to T$ where $T\leq \int_0^{g(0)}\frac{dv}{\phi(v)}$. With $g(0) = -\tilde L_fV(x(0))$ and the fact that the bound on $T$ is finite, we have that $\tilde L_fV(x(t))\to 0$ as $t\to T$, implying the solutions $x(\cdot)$ reach the largest invariant set $E\subset M\cap S$ within a finite time. The fixed-time convergence follows when the integral $\int_0^{g(0)}\frac{dv}{\phi(v)}$, denoting the settling-time, is bounded for all $g(0)\in \mathbb R_+$. 
\end{proof}

\section{Stability analysis for modified dynamics}\label{sec: AS FTS equivalence}
Now that we have established conditions for FTS and FxTS using LaSalle's invariance principle-like results, inspired by the modified gradient flow dynamics in \cite{garg2020fixed}, consider the modified system where the dynamics in \eqref{eq: sys} is normalized by powers of the norm of the function $f$ in the following manner to obtain the following modified dynamical system:
\begin{align}\label{eq: norm sys}
    \dot x = \begin{cases}
        \left(\frac{1}{\|f(x)\|^{p}}+\frac{1}{\|f(x)\|^{q}}\right)f(x) &  x\neq 0,\\
        0 & x = 0,
    \end{cases}
\end{align}
where $f:\mathbb R^n\to \mathbb R^n$ is continuously differentiable with $f(0) = 0$ and $f(x)\neq 0$ for all $x\neq 0$,  $0<p<1$ and $q<0$. 
Note that the above dynamics is well-defined and continuous for all $x\in \mathbb R^n$ when $f$ is continuous over $\mathbb R^n$ and $f(x) = 0 \iff x = 0$. Furthermore, using \cite[Proposition 1]{garg2022fixed}, it can be shown that the solutions of \eqref{eq: norm sys} are uniquely determined for all $t\geq 0$. We now state the following result, which establishes the FxTS of the origin \eqref{eq: norm sys} when the origin for \eqref{eq: sys} is GAS with similar assumptions as in \cite[Theorem 8]{cortes2006finite}. 


\begin{Theorem}\label{thm: FxTS second order}
    Assume that $\lim_{x\to \infty}\|f(x)\| = \infty$ and that either of the following conditions holds for the matrix function $H:\mathbb R^n\to \mathbb R^{n\times n}$ defined as $H(x) = -\left(\nabla f(x) + \nabla f(x)^T\right)$:
    \begin{itemize}
        \item[(i)] The matrix $H(x)$ is positive definite with its minimum eigenvalue $\lambda_0>0$ for all $x\neq 0$; or
        \item[(ii)] The matrix $H$ is positive semidefinite for all $x\in \mathbb R^n\setminus \{0\}$, the multiplicity of the eigenvalue $0$ is constant over $\mathbb R^n$ and for all $x\in \mathbb R^n$, the vector $f(x)$ is orthogonal to the eigenspace of $H(x)$ corresponding to the eigenvalue $0$, and the second smallest eigenvalue of $H$ is positive for all $x\in \mathbb R^n$, i.e., there exists $\lambda_0>0$ such that $\inf\limits_{x}\lambda_2(H(x)) = \lambda_0$.
    \end{itemize}
     Then, the origin for \eqref{eq: sys} is GAS and for \eqref{eq: norm sys} is FxTS. 
\end{Theorem}

\begin{proof}
    Consider a candidate Lyapunov function $V:\mathbb R^n\to \mathbb R$ as $V(x) = \|f(x)\|^2$ so that $V(x)>0$ for all $x\neq 0$ and $V(0) = 0$. Since $f$ is continuously differentiable, $V$ is continuously differentiable and its derivative along the trajectories of \eqref{eq: sys} reads 
    \begin{align*}
        \dot V(x) & = f(x)^T\nabla f(x) \dot x + \dot x^T\nabla f(x)^Tf(x)\\
        & = - f(x)^TH(x)f(x).
    \end{align*}
    We first provide a lower bound on $f(x)^TH(x)f(x)$ using a similar analysis as in \cite[Theorem 8]{cortes2006finite} under the assumptions on $H$. 
    It holds that $\lambda_2(H(x))\geq \lambda_0$ for all $x$ per definition of $\lambda_0$. Now, it also holds that for any $u\in \mathbb R^n$
    \begin{align*}
        u^TH(x)u \geq \lambda_0\|u-\pi_{H_0(H(x))}u\|^2,
    \end{align*}
    where $\pi_S(b)$ denotes the projection of the vector $b\in \mathbb R^n$ on the set $S\subset\mathbb R^n$ and $H_0(A)$ denotes the eigenspace corresponding to the zero eigenvalue of a matrix $A\in \mathbb R^{n\times n}$. 
    
    Per either of the assumptions on $H$, it holds that 
    $f(x)^TH(x)f(x)\geq \lambda_0\|f(x)\|^2$ for all $x$. Thus, we obtain:
    \begin{align*}
        \dot V(x) = - f(x)^TH(x)f(x)\leq - \lambda_0\|f(x)\|^2 = -\lambda_0V(x),
    \end{align*}
    for all $x\in \mathbb R^n$. As a result, we obtain that the origin of \eqref{eq: sys} is GAS. 
    
    Now, the derivative of the same function $V$ along the trajectories of \eqref{eq: norm sys} reads:
    \begin{align*}
        \dot V(x) & = f(x)^T\nabla f(x) \dot x + \dot x^T\nabla f(x)^Tf(x)\\
        & = - f(x)^TH(x)f(x)\left(\frac{1}{\|f(x)\|^p}+\frac{1}{\|f(x)\|^q}\right).
    \end{align*}
    Thus, we obtain
    \begin{align*}
        \dot V(x) & = - f(x)^TH(x)f(x)\left(\frac{1}{\|f(x)\|^p}+\frac{1}{\|f(x)\|^q}\right)\\
        & \leq - \lambda_0\|f(x)\|^2\left(\frac{1}{\|f(x)\|^p}+\frac{1}{\|f(x)\|^q}\right)\\
        & = -\lambda_0(\|f(x)\|^{2-p}+\|f(x)\|^{2-q})\\
        & = -\lambda_0 \left(V(x)^\frac{2-p}{2}+V(x)^\frac{2-q}{2}\right).
    \end{align*}
    Hence, using \cite[Lemma 2]{polyakov2012nonlinear}, we have that the origin is FxTS for \eqref{eq: norm sys} with settling-time function satisfying
    \begin{align*}
        T\leq \frac{1}{\lambda_0(1-\alpha)} + \frac{1}{\lambda_0(\beta-1)},
    \end{align*}
    where $0<\alpha = \frac{2-p}{2}<1$ and $\beta = \frac{2-q}{2}>1$. 
\end{proof}

\begin{Remark}
The assumptions on $H$ in Theorem \ref{thm: FxTS second order} are inspired from the results in \cite{cortes2006finite} where the authors relax the condition of positive-definiteness requiring a uniform positive lower bound on the smallest eigenvalue of a concerned matrix (in this case, $H$) with positive semi-definiteness with additional condition that the dynamical function $f$ has a zero projection in the null space of the concerned matrix. Here, we use the same setup, but for the modified system \eqref{eq: norm sys}, to prove FxTS. Note also that the first condition on $H$ in Theorem \ref{thm: FxTS second order} can also be used to establish (local) exponential stability of the origin for \eqref{eq: sys}. As a result, this result concretely proves the equivalence of the stability notions ES, AS, FTS, and FxTS. 
\end{Remark} 

Now, we state our final result for the case when $f$ is a gradient field, i.e., there exists a scalar function $V:\mathbb R^n\to \mathbb R$ such that $\nabla V(x) = -f(x)$. In this case, we define a piecewise dynamics given as
\begin{align}\label{eq: norm piecewise}
    \dot x = \begin{cases}
        c\frac{f(x)}{\|f(x)\|^\alpha} & x\neq 0, \; \|f(x)\|\leq 1, \\
        c\frac{f(x)}{\|f(x)\|^\beta} & \|f(x)\|> 1, \\
        0 & x = 0,
    \end{cases}
\end{align}
where $f$ has the same conditions as in \eqref{eq: norm sys}, with $\alpha\in (0, 1)$,  $\beta<0$ and $c>0$ given as
\begin{align*}
    c = \left(\frac{2-\beta}{2-\alpha}\right)^{\frac{1}{p-q}},
\end{align*}
where $p = \frac{2-2\alpha}{2-\alpha}$ and $q = \frac{2-2\beta}{2-\beta}$. Note that $0<p<1$ for $\alpha\in (0, 1)$ and $q>1$ for $\beta<0$. It is easy to prove that the right-hand side of \eqref{eq: norm piecewise} is continuous on $\mathbb R^n$ and Lipschitz on $\mathbb R^n\setminus\{0\}$. The parameter $c>0$ here comes in handy to construct a continuous function $\phi$ that satisfies the conditions given in Theorem \ref{thm: FxTS LaSalle second order}. We can now state the following result. 

\begin{Theorem}\label{thm: FxTS second order gradient field}
    Assume that $f$ is a gradient field with $\int_0^xf(y)^Tdy\to -\infty$ as $x\to \infty$. Let either of the following conditions holds for the matrix function $H:\mathbb R^n\to \mathbb R^{n\times n}$ defined as $H(x) =-\frac{1}{2}(\nabla f(x) + \nabla f(x)^T)$:
    \begin{itemize}
        \item[(i)] The matrix $H(x)$ is positive definite with its minimum eigenvalue $\lambda_0>0$ for all $x\neq 0$; or
        \item[(ii)] The matrix $H(x)$ is positive semidefinite for all $x\in \mathbb R^n\setminus \{0\}$, the multiplicity of the eigenvalue $0$ is constant over $\mathbb R^n$ and $f(x)$ is orthogonal to the eigenspace of $H(x)$ corresponding to the eigenvalue $0$ for all $x$, and the second smallest eigenvalue of $H$ is positive for all $x\in \mathbb R^n$, i.e., $\inf_{x}\lambda_2(H(x)) = \lambda_0>0$
    \end{itemize}
    Then, the origin for \eqref{eq: norm piecewise} is FxTS. 
\end{Theorem}

\begin{proof}
{
    The proof is based on satisfying all the conditions of Theorem \ref{thm: FxTS LaSalle second order}. 
    Inspired from the \textit{variable gradient method} of constructing Lyapunov functions (see, e.g., \cite[pp 120]{khalil2002nonlinear}), consider a Lyapunov-like function $V:\mathbb R^n\to \mathbb R$ defined as $V(x) = -\int_0^xf(y)^Tdy$. We consider the cases $\|f(x)\|\leq 1$ and $\|f(x)\|>1$ separately. First, consider the case when $\|f(x)\|\leq 1$. Since $f$ is a gradient field, this integral is path-independent and it holds that $V$ is continuously differentiable and $\dot V$ along the trajectories of \eqref{eq: norm piecewise} is given as:
    \begin{align*}
        \dot V(x) & = -f(x)^T\left(\frac{1}{\|f(x)\|^{\alpha}}\right)f(x) = -c\|f(x)\|^{2-\alpha},
    \end{align*}
    implying $\dot V(x)\leq 0$ for all $x$ such that $\|f(x)\|\leq 1$, and $\dot V(x) = 0$ when $f(x) = 0$, or equivalently, $x = 0$. Next, let $g$ denote the RHS of \eqref{eq: norm piecewise} and consider $ L_g(\dot V(x))$ for the case when $\|f(x)\|\leq 1$ given as
    \begin{align*}
        L_g\left(\dot V(x)\right)= & -c^2(2-\alpha)\|f\|^{1-\alpha}\frac{f(x)^T}{\|f(x)\|}\nabla f(x)\frac{f(x)}{\|f(x)\|^\alpha}\\ 
        & = \frac{c^2}{2}(2-\alpha)f(x)^TH(x)f(x)\|f(x)\|^{-2\alpha},
    \end{align*}
    where the last equality is obtained using the fact that $a^TBa = a^TB^Ta = \frac{1}{2}a^T(B+B^T)a$ for any $a\in \mathbb R^n$ and $B\in \mathbb R^{n\times n}$. Similar to the analysis in Theorem \ref{thm: FxTS second order}, we can show that it holds that $f(x)^TH(x)f(x)\geq \lambda_0\|f(x)\|^2$ for all $x$. Hence, we obtain:
    \begin{align*}
        \tilde L_g\left(\dot V(x)\right)\geq & \frac{c^2}{2}(2-\alpha)\lambda_0(\|f(x)\|^{2-2\alpha} \\
        \geq &\frac{c^2}{2}(2-\alpha)\lambda_0 \left(\frac{-\dot V(x)}{c}\right)^p.
    \end{align*}
    Repeating the same analysis again for the case $\|f(x)\|>1$, we obtain
    \begin{align*}
        \tilde L_g\left(\dot V(x)\right)\geq  \frac{c^2}{2}(2-\beta)\lambda_0 \left(\frac{-\dot V(x)}{c}\right)^q,
    \end{align*}
     Let $\phi:\mathbb R\to \mathbb R$ be defined as
     \begin{align*}
         \phi(r) = \begin{cases}
             \frac{c^{2-p}}{2}(2-\alpha)\lambda_0 \left(r^p\right) & r\leq 1, \\
             \frac{c^{2-q}}{2}(2-\beta)\lambda_0 \left(r^q\right) & r>1,
         \end{cases}
     \end{align*}
     so that
    \begin{align*}
        \tilde L_g\left(\dot V(x)\right)\geq \phi(-\dot V(x)) \quad \forall x\in \mathbb R^n. 
    \end{align*}
    It is easy to show that $\phi$ is continuous at $r = 1$ thanks to the choice of the parameter $c>0$. 
    With this definition of the function $\phi$, we have that $r\phi(r)> 0$ for all $r\neq 0$ and $\phi(0) = 0$. We also have
    \begin{align*}
        \int_0^{r_0}\frac{dr}{\phi(r)} & =  \int_0^{1}\frac{dr}{\phi(r)}+\int_1^{r_0}\frac{dr}{\phi(r)},
    \end{align*}
    for all $r_0\in \mathbb R_+$. 
    Since $\phi$ is a monotonic function, when $0<r_0<1$, it follows that
    \begin{align*}
        \int_0^{r_0}\frac{dr}{\phi(r)} & = \int_0^{r_0}\frac{2 c^{p-2}dr}{(2-\alpha)\lambda_0r^p} \\
        & \leq \int_0^{1}\frac{2c^{p-2}dr}{(2-\alpha)\lambda_0r^p} \\
        & =  \frac{2c^{p-2}}{(2-\alpha)\lambda_0(1-p)}.
    \end{align*}
    For $r_0>1$, we obtain
    \begin{align*}
        \int_1^{r_0}\frac{dr}{\phi(r)} & = \int_1^{r_0}\frac{2c^{q-2}dr}{(2-\beta)\lambda_0r^q}\\
        & \leq \int_1^{\infty}\frac{2c^{q-2}dr}{(2-\beta)\lambda_0r^q}\\
        & = \frac{2c^{q-2}}{(2-\beta)\lambda_0(q-1)}
    \end{align*}
    As a result, we obtain that
    \begin{align*}
        \int_0^{v_0}\frac{dv}{\phi(v)} &\leq \frac{2c^{p-2}}{(2-\alpha)\lambda_0(1-p)} + \frac{2c^{q-2}}{(2-\beta)\lambda_0(q-1)},
    \end{align*}
    for all $v_0\in \mathbb R_+$. Hence, all the conditions of Theorem \ref{thm: FxTS LaSalle second order} are satisfied. As a result, it holds that $x(t)\to M$ within a fixed time for all $x(0)$. However, $M = \{x\; |\; L_fV(x) = 0\}$, which is just the singleton set $M = \{0\}$. Hence, the solutions of \eqref{eq: norm piecewise} reach the origin within a fixed time, or the origin is FxTS for \eqref{eq: norm piecewise}. }
\end{proof}
\begin{Remark}
    It is not necessary to consider piecewise dynamics. A similar, albeit more convoluted analysis, can be carried out for the same dynamics as in \eqref{eq: norm sys} under the assumptions of Theorem \ref{thm: FxTS second order gradient field}. However, we adopt this mechanism for piecewise analysis, inspired in part by the original proof of FxTS in \cite{polyakov2012nonlinear}, to keep the exposition simple. 
\end{Remark}

\section{Conclusions}\label{sec: conclusions}
We presented new results on Lyapunov-like theorems using non-smooth analysis for finite- and fixed-time stability of dynamical systems. In particular, we provided conditions similar to the ones used in LaSalle's invariance principle for convergence to an invariant set within a fixed amount of time. Then, we proved that a particular normalization or scaling of a system, whose equilibrium is asymptotically stable, leads to fixed-time stability, thereby connecting these notions of stability and proving the equivalence between asymptotic, finite- and fixed-time stability for a general class of dynamical systems. 
We are currently exploring the utility of these results for analyzing algorithms for solving non-smooth optimization problems. The next steps for this line of work are to explore how these results can be extended and used for control systems, both for design and analysis. 

\section{Acknowledgments}
The author acknowledges Dr. Rohit Gupta for several fruitful discussions.

\bibliographystyle{IEEEtran}
\bibliography{myreferences}

\end{document}